\documentclass[preprint,12pt]{elsarticle}
\usepackage[pagebackref=true,colorlinks,linkcolor=blue,citecolor=magenta]{hyperref}
\usepackage[margin=1in]{geometry}  
\usepackage{graphicx}              
\usepackage{amsmath}               
\usepackage{amsfonts}              
\usepackage{amsthm}     
\usepackage{amssymb}
\usepackage{tikz}  
\usepackage{subfigure}         
\usepackage{amsmath}
\usepackage{times}
\usepackage{array, amssymb, amsmath, graphics}
\newtheorem{theorem}{Theorem}
\newtheorem{claim}{Claim}

\newtheorem*{lemmax}{Lemma}

\newtheorem*{counterexample}{Counterexample to ``Theorem \ref{D}''}

\newtheorem{pretheorema}{{\bf Theorem}}

\newtheorem{prelem}{{\bf Theorem}}

\newenvironment{lem}{\begin{prelem}{\hspace{-0.5
               mm}{\bf}}}{\end{prelem}}
\newtheorem{precor}{{\bf Corollary}}

\newtheorem{prelemlem}{{\bf Lemma}}

\begin{document}
\begin{frontmatter}
\title{On the relation between connectivity, independence and generalized caterpillars}
\author{M. Pedramfar}
\ead{pedramfar@dena.sharif.edu}
\author{M. Shokrian\corref{cor1}}
\ead{mshokrian@dena.sharif.edu}
\author{M. Tefagh\corref{cor2}}
\ead{mtefagh@dena.sharif.edu}
\cortext[cor1]{Corresponding author}
\address{Department of Mathematical Sciences, Sharif University of Technology, P.O. Box 11155-9415, \\ Tehran, I.R. IRAN}

\hyphenation{Mathematical}
\hyphenation{Hypercubes}
\hyphenation{independence}
\hyphenation{code}
\hyphenation{Furthermore}
\hyphenation{diameter}
\hyphenation{homomorphism}
\hyphenation{isomorphism}
\hyphenation{either}
\hyphenation{delsarte}

\begin{abstract}
A spanning generalized caterpillar is a spanning tree in which all vertices of degree more than two are on a path. In this note, we find a relation between the existence of spanning generalized caterpillar and the independence and connectivity number in a graph. We also point out to an error in a ``theorem'' in the paper ``Spanning spiders and light-splitting switches'', by L. Gargano et al.  in Discrete Math. (2004), and find out a relation between another mentioned theorem and the existence of spanning generalized caterpillar.
\end{abstract}
  
\end{frontmatter} 
Here, we consider only finite connected graphs without loops or multiple edges. For standard graph-theoretic terminologies not explained in this note, we refer to \cite{MR1367739}. 

We denote the $\it{degree}$ of a vertex $x$ in a graph $G$ by $d_G(x)$, the $\it{independence \ number}$ by $\alpha(G)$ and its $\it{connectivity \ number}$ by $\kappa(G)$. In a tree $T$, we call a vertex $v \in V(T)$ a $\it{branch \ vertex}$ when $d_T(v) > 2$.

As defined in \cite{Shrestha:2011:BCB:2033094.2033122}, a tree is called a $\it{generalized \  caterpillar}$ if all vertices of degree more than two are on a path or equivalently a caterpillar in which $\textit{hairs}$ edges incident to the spine, are replaced by paths. 

A motivation to define such a tree comes from the definition of $\it{spider}$, i.e. a star in which hairs are replaced by paths. Obviously a generalized caterpillar is also a generalization of spider.

In earlier results, it is shown that:
\begin{lem}
\emph{(\cite{Monien:1986:BMP:12903.12904})} The bandwidth problem is NP-complete for generalized caterpillars of hair length at most $3$.
\end{lem}

But we are interested in  the conditions under which a graph has a spanning generalized caterpillar (or for simplicity an $\textbf{SGC}$). In order to find some sufficient conditions, we concentrated on the special case of spiders. In this way we looked at the paper \cite{MR2074842}, where it is discussed about the different aspects of the existence of spanning spiders in a graph:
\begin{lem}
\emph{(Proposition 1 of \cite{MR2074842})} It is NP-complete to decide whether a given graph $G$ admits a spanning spider.
\end{lem}
So finding necessary and sufficient conditions for the existence of SGC in a graph does not seem to be an easy task! One of the results which was particularly interesting for us is the following:

\begin{lem}
\label{C}
\emph{(Theorem 8 of \cite{MR2074842})} Let $G$ be a connected graph. Then $s(G) \le 2 \lceil \frac{\alpha(G)}{\kappa(G)} \rceil - 2$, where $s(G)$ is the minimum number of branch vertices in a spanning tree in $G$. 
\end{lem}

But we found that the following wrong theorem was used to prove Theorem \ref{C} which is:
\begin{lem}
\label{D}
\emph{(\cite{MR2074842})} Vertices of any graph G can be covered by at most $\lceil \frac{\alpha(G)}{\kappa(G)} \rceil$ vertex disjoint paths.
\end{lem}
We checked the reference \cite{MR0297600} to which the authors have referred, and we could not find that used result in there. Indeed, we have a counterexample to ``Theorem \ref{D}'':

\begin{counterexample}
Let $G=K_{m,2m}$. We have $\alpha(G)=2\kappa(G)=2m$ but we can not cover the vertices of $G$ by at most two vertex disjoint paths $P$ and $Q$.
\end{counterexample}
\begin{proof}
If $P=\emptyset $, it means that $G$ has a hamiltonian path which is obviously wrong. So $P$ and $Q$ are not empty. Since $G$ is connected there is an edge $e$ which is incident with a vertex in $P$ and also a vertex in $Q$. Adding this edge leads to a spanning tree $T$ which has all vertices of degree at most $2$, except for at most two vertices of degree $3$. $G$ is bipartite, so the edges of $T$ have exactly one vertex in the part with $m$ vertices. But we have:  
\begin{align*}
|E(T)| \le 2+2+\cdots +2+3+3=2(m-2)+6=2m+4.
\end{align*}

On the other hand, since $T$ is a spanning tree we know that $|E(T)|=|V(G)|-1=3m-1$ which is greater than $2m+4$ for $m > 5$. So Theorem \ref{D} is wrong. 
\end{proof}

In fact, we have:
\begin{lem}
\emph{(\cite{MR1271279})} Vertices of any graph G can be covered by at most $\lceil \frac{\alpha(G)}{\kappa(G)} \rceil$ cycles.
\end{lem}

But what is the relation of Theorem \ref{C} to the existence of SGC? At first glance, the relation between connectivity, independence and existence of certain kind of spanning trees might not be obvious. But there exist many theorems discussing about this relation \cite{MR2746831}. 

The following theorem answers the above question:
\begin{theorem}
\label{1}
If $s(G) \le \kappa(G)$, then $G$ has a spanning generalized caterpillar.
\end{theorem}

\begin{proof}
We use the following well-known lemma to prove it.
\begin{lemmax}
\emph{(\cite{MR1367739})}
In any graph $G$, any set with $\kappa(G)$ vertices can be covered by a cycle. 
\end{lemmax}
The proof of the lemma given above can be found in \cite{MR1367739}, page 170. Now since there exists a spanning tree $T$ with at most $\kappa(G)$ vertices, we can add a cycle $C$ to the edges of $T$ so that in the new spanning subgraph $T'$ all vertices of degree more than $2$ are on that cycle. Now delete one arbitrary edge from $C$, like $e$, and from each cycle in $T'$ delete one arbitrary edge which is not in $C-e$ (as $C-e$ is a path we can do this). At last, we obtain a spanning tree in which all vertices of degree more than $2$ are on a path : $C-e$. So we have an SGC.
\end{proof}

Now if we consider graphs for which $\alpha(G) \le \frac{(\kappa(G))^2+\kappa(G)}{2}$, then by Theorem \ref{C} we have:
\begin{align*}
s(G) \le 2\lceil \frac{\frac{(\kappa(G))^2+\kappa(G)}{2}}{\kappa(G)} \rceil-2 =\left\{\begin{array}{ll} \kappa(G) ,&\mbox{ if }
2 | \kappa(G) \\ \kappa(G)-1 , & \mbox{ if } 2|\kappa(G)+1.  \end{array}\right. 
\end{align*}
So as a corollary of Theorem \ref{1}, we can deduce that for all graphs with $\alpha(G) \le \frac{(\kappa(G))^2+\kappa(G)}{2}$ we have an SGC.

Now we try to construct graphs which do not have SGC:
\begin{theorem}
\label{2}
There exists a graph G having no spanning generalized caterpillar as in the following construction:
\begin{enumerate}
\item
Consider an independent set of vertices $S=\{v_1,\ldots,v_m\}$.
\item
Consider $m+2$ graphs $G_i=K_{2m+1,m} \ , \ 1 \le i \le m+2$. From each $G_i$ select $m$ arbitrary vertices from the part with $2m+1$ vertices like $\{t_{i,1}, \ldots,t_{i,m}\}$ and for each $1 \le i \le m+2$ add the edge $t_{i,j}v_r$ for all $1 \le r,j \le m$.

\end{enumerate}

\end{theorem}
\begin{figure}[h]

\ \ \ \ \ \ \ \ \ \ \ \ \ \ \  \includegraphics[scale=0.7]{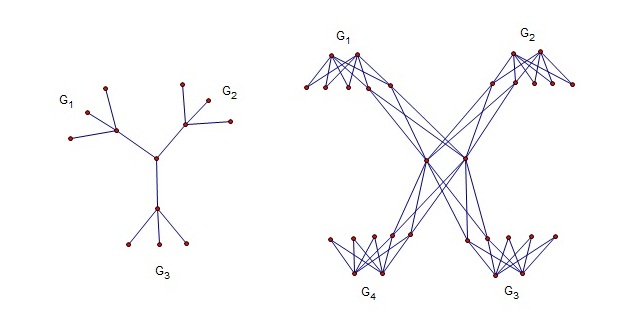}
\caption{Cases $m=1$ and $m=2$}
\end{figure}
\renewcommand\qedsymbol{$\blacksquare$}

\begin{proof}
We have shown the cases $m=1$ and $m=2$ in Figure 1. We show by contradiction that G does not have an SGC. Suppose that we have an SGC like $T$. Consider $P$, the set that contains all branch vertices. Let $A_i$ be the part with $m$ vertices in $G_i$ and $B_i$ be the other part which has $2m+1$ vertices.
\begin{claim}
for each $i$, $1 \le i \le m+2$ we have $ G_i \cap P \neq \emptyset$.
\end{claim}

$\it{Proof \ of \ the \ claim \ 1.}$ If for some $i$ we have $G_i \cap P = \emptyset$, by considering the definition of GC, it means that all vertices of $G_i$ are covered by, let us say, $n$ vertex disjoint paths. And also, each of these paths has at least one of its ends, in the set of vertices $\{t_{i,1},\ldots,t_{i,m}\}$, because these are the only vertices which are connected to the vertices outside of $G_i$. Hence we get $n \le m$. 

If we denote the paths by $\{P_1,\ldots,P_n\}$, then obviously as the graph $G_i$ is bipartite  we have $|P_u \cap B| \le |P_u \cap A|+1$, for each $u$, $1 \le u \le n$. So we obtain:
\begin{align*}
2m+1=|B|=\sum_{u=1}^{n} |P_u \cap B| \le \sum_{u=1}^{n} (|P_u \cap A|+1)=|A|+n \implies m+1 \le n,
\end{align*}
which is a contradiction with $n \le m$. Therefore for each $i$, $1 \le i \le m+2 $ we have $ \ G_i \cap P \neq \emptyset$.   $ \qquad \qquad \qquad \qquad \qquad \qquad \qquad \qquad \qquad \qquad \qquad \qquad \qquad \qquad \qquad \ \  \ \square$

Define $Q$ to be the shortest path in $T$ that contains all branch vertices. Now we claim the following and the contradiction is immediate:
\begin{claim}
We have more than $2m$ edges which are in $Q$ and also are incident to a vertex in $S$.
\end{claim}

$\it{Proof \ of \ the \ claim \ 2.}$ In the last claim we proved that for each $i$, $ 1 \le i \le m+2 $ we have $G_i \cap Q \neq \emptyset$ . Obviously $G_i \cap Q$ is a set of vertex disjoint paths. Take one of them like $L$. Now we look at the ends of $L$, namely $x$ and $y$. There exist three cases:
\begin{itemize}
\item[(i)]
 $x$ or $y$ (and not both) is at the end of $Q$.
\item[(ii)]
both of $x,y$ are vertices of degree $2$ in $Q$.
\item[(iii)]
$x,y$ are the ends of $Q$.
\end{itemize}

In the case (i), if $y$ is at the end of $Q$ then since $x$ has degree $2$ in $Q$, one of two edges is in $E(G_i)$ and the other is an edge between $G_i$ and $S$. On the other hand as $Q$ has two ends, this case can happen for at most two $G_i$'s.

In the case (ii), both of the vertices have the same condition of $x$ in the last case. So there are two edges between $G_i$ and $S$ which are in $Q$.

As for the case (iii), it can not happen, because $Q$ is connected and has vertices in at least two different $G_i$'s. So there must exist an edge in $Q$ but not in $L$ and incident to one of the vertices of $L$ so that $L$ can be connected to other parts of $Q$ in other $G_i$'s.

Hence by considering the fact that (i) happens at most two times there exist at least:
\begin{align*}
2+2+\cdots +2+1+1=2(m+2-2)+2=2m+2 > 2m
\end{align*}
edges between each $G_i$ and $S$ which are in $Q$.   $\qquad \qquad \qquad \qquad \qquad \qquad \qquad \qquad \ \ \ \ \ \square$

 But $|S|=m$ and each vertex in $Q$ has degree at most $2$. As $\frac{2m+2}{|S|}>2$, there is a vertex $v \in S \cap Q$ which has degree at least $3$ in $Q$. Thus $G$ does not have any SGC.
\end{proof}
It is easy to see that $\alpha(G)=(2m+1)(m+2)$ and $\kappa(G)=m$. Hence, we have proved that if $\alpha(G) \ge (2\kappa(G)+1)(\kappa(G)+2)$, then we may not have SGC. We don't know whether Theorem \ref{C} is true or wrong but if we can construct `better' graphs which have no SGC, i.e. a graph for which the ratio $\frac{\alpha(G)}{\kappa(G)}$ is less than $\frac{\kappa(G)}{2}$, then we may be able to find a counterexample for $s(G) \le 2\lceil \frac{\alpha(G)}{\kappa(G)} \rceil-2$.
\renewcommand\qedsymbol{$\square$}

At last, we provide a bound for $\alpha(G)$ in terms of $\kappa(G)$ which guarantees the existence of SGC:
\begin{theorem}
If $\alpha(G) \le 2\kappa(G)+1$, then $G$ has an spanning generalized caterpillar.
\end{theorem}
\begin{proof}
Define $V_3^T=\{v| v \in T, d_T(v)=3\}$ and let a $k-$tree be a tree with $\Delta(T) \le k$. We use the following result which appears in \cite{MR2506391}:
\begin{lem}\label{F}
\emph{(\cite{MR2506391})} If $\alpha(G) \le \kappa(G)+n+1$, where $0 \le n \le \kappa(G)$, then we have a spanning $3-$tree $T$ with $|V_3^T| \le n$.  
\end{lem}
By using Theorem \ref{F} for the case $n=\kappa(G)$ and Theorem \ref{1} we have an SGC. Since before adding the cycle $C$ mentioned in Theorem \ref{1}, the degree of each vertex was at most $3$, we get an SGC with maximum degree at most $5$.
\end{proof}
\newpage
\section*{Acknowledgements}
The authors greatly appreciate Prof. E.S. Mahmoodian for his corrections and many valuable comments. We also would like to acknowledge Prof. S. Akbari for introducing us this problem.

\section*{References}

%\bibliographystyle{elsarticle-num}
%\bibliography{PedramShT2}

\end{document}